\documentclass[a4paper]{amsart}

\PassOptionsToPackage{usenames,dvipsnames}{xcolor}
\usepackage{multirow,mathabx}
\usepackage[english]{babel}
\usepackage{fancyhdr}

\usepackage{mathtools}
\usepackage{amsthm,amssymb,latexsym,url,cancel}
\usepackage{pdfpages}

\usepackage{enumerate}
\makeatletter
\let\@@enum@org\@@enum@
\def\@@enum@[#1]{\@@enum@org[\normalfont #1]}
\makeatother

\def\co{\colon\thinspace}

\DeclarePairedDelimiter{\form}{\langle}{\rangle}

\DeclarePairedDelimiter{\abss}{\lVert}{\rVert}

\newcommand\PSL{\operatorname{PSL}}

\newcommand\ba{\begin{align*}}
\newcommand\ea{\end{align*}}
\newcommand\be{\begin{enumerate}}
\newcommand\ee{\end{enumerate}}
\newcommand\bp{\begin{proof}}
\newcommand\ep{\end{proof}}
\newcommand\bpp{\begin{prop}}
\newcommand\epp{\end{prop}}
\newcommand\bpb{\begin{prob}}
\newcommand\epb{\end{prob}}
\newcommand\bd{\begin{defn}}
\newcommand\ed{\end{defn}}
\newcommand\bh{\begin{hint}}
\newcommand\eh{\end{hint}}

\numberwithin{equation}{section}

\newcommand\bC{\mathbb{C}}

\newcommand\bN{\mathbb{N}}
\newcommand\bR{\mathbb{R}}
\newcommand\bQ{\mathbb{Q}}
\newcommand\bZ{\mathbb{Z}}

\newcommand\VV{\mathcal{V}}
\newcommand\UU{\mathcal{U}}

\newcommand\Hom{\operatorname{Hom}}

\newcommand\Mod{\operatorname{Mod}}

\DeclareMathOperator\Homeo{Homeo}

\newcommand\yt{\widetilde}
\newcommand\sse{\subseteq}

\DeclareMathOperator\Out{Out}
\DeclareMathOperator\Aut{Aut}

\DeclareMathOperator\Diff{Diff}

\usepackage[margin=2.5cm]{geometry}

\def\thetitle{Small $C^1$ actions of semidirect products on compact manifolds}
\def\theauthors{{THE AUTHORS HERE}}
\usepackage{hyperref}
\hypersetup{
   hidelinks = true,
   colorlinks=false,
   plainpages,
   urlcolor=black,
   linkcolor=black
   pdftitle=   \thetitle,
   pdfauthor=  {\theauthors}
}

\newtheorem{thm}{Theorem}[section]
\newtheorem{lem}[thm]{Lemma}
\newtheorem{cor}[thm]{Corollary}
\newtheorem{prop}[thm]{Proposition}

\newtheorem{que}[thm]{Question}
\newtheorem*{claim*}{Claim}
\newtheorem*{assumption*}{Assumption}

\theoremstyle{remark}

\newtheorem*{hint}{Hint}

\theoremstyle{definition}
\newtheorem{defn}[thm]{Definition}
\newtheorem{prob}{Problem}[section]

\begin{document}
\title{\thetitle}
\date{\today}
\keywords{Groups acting on manifolds, hyperbolic dynamics, fibered $3$--manifold, $C^1$--close to the identity}

\author[C. Bonatti]{Christian Bonatti}
\address{CNRS \& Institut de Math\'ematiques de Bourgogne (IMB, UMR CNRS 5584),
Universit\'e Bourgogne Franche--Comt\'e,
9 av.~Alain Savary,
21000 Dijon, France.}
\email{Christian.Bonatti@u-bourgogne.fr}
\urladdr{http://bonatti.perso.math.cnrs.fr}

\author[S. Kim]{Sang-hyun Kim}
\address{School of Mathematics, Korea Institute for Advanced Study (KIAS), Seoul, 02455, Korea}
\email{skim.math@gmail.com}
\urladdr{http://cayley.kr}

\author[T. Koberda]{Thomas Koberda}
\address{Department of Mathematics, University of Virginia, Charlottesville, VA 22904-4137, USA}
\email{thomas.koberda@gmail.com}
\urladdr{http://faculty.virginia.edu/Koberda}

\author[M. Triestino]{Michele Triestino}
\address{Institut de Math\'ematiques de Bourgogne (IMB, UMR CNRS 5584),
Universit\'e Bourgogne Franche--Comt\'e,
9 av.~Alain Savary,
21000 Dijon, France.}
\email{Michele.Triestino@u-bourgogne.fr}
\urladdr{http://mtriestino.perso.math.cnrs.fr}

\begin{abstract}
Let $T$ be a compact fibered $3$--manifold, presented as a mapping torus of a compact, orientable surface $S$ with monodromy $\psi$,
and let $M$ be a compact Riemannian manifold.
Our main result is that if the induced action $\psi^*$ on $H^1(S,\bR)$ has no eigenvalues on the unit
circle, then there exists a neighborhood $\mathcal U$ of the trivial action in the space of $C^1$ actions of $\pi_1(T)$ on $M$ such that any
action in $\mathcal{U}$ is abelian. We will prove that the same result holds in the generality of an infinite cyclic extension
of an arbitrary finitely generated group $H$, provided that the conjugation action of the cyclic group on $H^1(H,\bR)\neq 0$ has no
eigenvalues of modulus one.
We thus generalize a result of A. McCarthy, which addressed the case of abelian--by--cyclic groups
acting on compact manifolds.

\smallskip
{\noindent\footnotesize \textbf{MSC\textup{2020}:} Primary 37C85. Secondary 20E22, 57K32.}
\end{abstract}

\maketitle


\section{Introduction}

In this paper, we consider smooth actions of finitely generated--by--cyclic groups on compact manifolds, motivated by the study
of fibered hyperbolic $3$--manifold groups. We let $S$ be a compact, orientable surface of
negative Euler characteristic, possibly with boundary. Thus, the fundamental group $\pi_1(S)$ is either a finitely generated free group, or the fundamental group
of a closed surface of genus $g$ for some $g\geq 2$. If $\phi\in\Homeo(S)$ is a (possibly orientation reversing)
homeomorphism, then we may form $T=T_{\phi}$,
the mapping torus of $\phi$. We have that the fundamental group $\pi_1(T)$ fits into a short exact sequence of the form \[1\to \pi_1(S)\to\pi_1(T)\to\bZ\to 1,\]
where the conjugation action of $\bZ$ on $\pi_1(S)$ is by the induced action of $\phi$. It is well--known that,
up to an inner automorphism of $\pi_1(S)$, this action depends only
on the homotopy class of $\phi$, and is therefore an invariant of the (extended) mapping class of $\phi$. It follows that the
 isomorphism type of $\pi_1(T)$ depends only
on the mapping class of $\phi$. 

We will be particularly interested in the action $\phi^*$ on  the real cohomology of the fiber, and especially
in the case where the induced map $\phi^*:H^1(S,\bR)\to H^1(S,\bR)$ is hyperbolic.
Here, the induced automorphism $\phi^*$ is said to be \emph{hyperbolic} if $H^1(S,\bR)\neq 0$ and if every eigenvalue of
$\psi^*$ has modulus different from one. More generally, an automorphism of a nonzero, finite dimensional real vector space is hyperbolic
if it has no eigenvalues of modulus one.

Examples of fibered $3$--manifolds with  hyperbolic monodromies
include all compact $3$--manifolds admitting sol geometry, as well as many fibered  hyperbolic manifolds. For example, the figure eight knot
complement fibers over the circle with  a punctured  torus as the fiber, and the monodromy given by an automorphism acting hyperbolically
on  the homology  of the torus.

Fibered $3$--manifold groups  arising  from mapping classes acting hyperbolically on the homology of the fiber
 fall into a much larger class of groups which we will be able to investigate with our methods. 
Here and throughout, we will let $M$ be a compact Riemannian manifold. 
Recall that a short exact sequence of finitely generated groups 
\[1\to H\to G\to\bZ\to 1\]
naturally determines $\psi\in \Out(H)$, and hence, induces a unique linear automorphism $\psi^*$ of $H^1(H,\bR)$.
Abstractly as groups, we have that $G$ is isomorphic to the semidirect product \[G\cong H\rtimes_{\psi} \bZ,\] where the outer automorphism $\psi$ is given by the conjugation
action of $\bZ\cong G/H$ on $H$. 
We will study $\Hom(G,\Diff^1(M))$, the space of $C^1$ actions of $G$ on $M$, in the case that $\psi^*$ is hyperbolic.

\subsection{Main result}
We will use the symbol $1$ to mean the identity map, the trivial group, the identity group element or the real number 1 depending on the context, as this will not cause confusion.  
The principal result of this paper is the following:

\begin{thm}\label{thm:main}
Suppose we have a short exact sequence of finitely generated groups
\[1\to H\to G\to\bZ\to 1,\]
which induces a hyperbolic automorphism $\psi^*$ of $H^1(H,\bR)$. Then there exists a neighborhood $\UU\sse\Hom(G,\Diff^1(M))$ of the trivial representation such that $\rho(H)=1$ for all representations $\rho\in\UU$.
\end{thm}

Thus, sufficiently small actions of $G$ on compact manifolds necessarily factor through cyclic groups, provided the automorphism defining the extension $G$ is hyperbolic on cohomology. The reader is directed to Subsection~\ref{subsec:space} for a discussion of the topology on $\Hom(G,\Diff^1(M))$. Theorem~\ref{thm:main} may be viewed as an analogue of a result of A.~McCarthy~\cite{McCarthy}, who proved a statement with the same conclusion for abelian--by--cyclic groups (the fundamental groups of compact 3--manifolds admitting sol geometry fall in this class).

For certain manifolds and with certain natural hypotheses, abelian--by--cyclic group actions by diffeomorphisms enjoy rather strong rigidity properties~\cite{HurtXue19}.

Note that if $H$ is a left-orderable group then it is not difficult to find faithful actions of $G$ by homeomorphisms of the interval $[0,1]$
which are arbitrarily $C^0$--close to the identity, so the $C^1$ regularity assumption in Theorem~\ref{thm:main} is essential.

By applying Theorem~\ref{thm:main} to the above short exact sequence for a fibered 3-manifold group, we obtain the following result.

\begin{cor}\label{cor:main-3M}
Let $S,\phi$, and $T$ be as above. If $\phi$ induces a hyperbolic automorphism of $H^1(S,\bR)$, then there exists a neighborhood
$\mathcal{U}$ of the trivial representation in $\Hom(\pi_1(T),\Diff^1(M))$ such that $\rho(H)=1$ for all representations $\rho\in\UU$.
\end{cor}

The hypotheses in Theorem~\ref{thm:main} may be contrasted with the following result of Bonatti--Rezaei~\cite{br2019},
which generalizes some work
of Farb--Franks~\cite{FF2003} and Jorquera~\cite{Jorquera}, and is closely related to results of Navas~\cite{NavasComp14} and
Parkhe~\cite{Parkhe16}.

\begin{thm}[Bonatti--Rezaei]\label{thm:br2019}
Every finitely generated, residually torsion--free nilpotent group $G$ admits a faithful representation $\rho\colon G\to\Diff^1([0,1])$
that is $C^1$--close to the identity.
\end{thm}

Here, a group is~\emph{residually torsion--free nilpotent} if every nontrivial element $g\in G$ survives in a torsion--free nilpotent quotient of $G$.
A representation $\rho\in\Hom(G,\Diff^1(M))$ is said to be \emph{$C^1$--close to the identity} if for every $\epsilon>0$, there is an element
$h=h_{\epsilon}\in \Diff^1(M)$ such that $h$ conjugates $\rho$ into an $\epsilon$--neighborhood of the trivial representation of $G$, in the $C^1$--topology on $\Hom(G,\Diff^1(M))$.

It is not difficult to check that if the group $G$ is as in the statement of Theorem~\ref{thm:main} then the only torsion--free nilpotent quotient admitted
by $G$ is $\bZ=G/H$. Thus, the hyperbolicity of the map $\psi^*$ plays a crucial role in the dynamics of the group $G$.

\subsection{Unipotent monodromy maps and virtually special groups}

An essential feature of Theorem~\ref{thm:main} is its ``unstable" nature, in the sense that it does not remain true after passing to finite
index subgroups of $G$. Indeed, we have the following fact, which follows fairly easily from known results:

\begin{prop}\label{prop:virtually special}
Let $N$ be a hyperbolic 3-manifold with finite volume.
Then a finite index subgroup of the fundamental group $\pi_1(N)$ admits a faithful representation $\rho$ into
$\Diff^1([0,1])$ such that $\rho$ is $C^1$--close to the identity.
\end{prop}

\begin{proof}
The essential point is that combining rather deep results of Agol and Wise with some combinatorial group theory arguments of Duchamp and Krob, 
one sees that the fundamental group $\pi_1(N)$ contains a finite index subgroup which is
residually torsion--free nilpotent, and hence admits a faithful representation into $\Diff^1([0,1])$
that is $C^1$--close to the identity.

In more detail, by the work of Agol and Wise~\cite{Agol2013,Wise2011}, there is a finite index subgroup $G_0<\pi_1(N)$ such that $G_0$
is \emph{special}. In particular, $G_0$ embeds in a \emph{right-angled
Artin group}; such a group is always residually torsion--free nilpotent~\cite{DK1992a}. See also the discussion in \cite[Chapter 5]{MR3444187}. Thus, the proposition follows from
Theorem~\ref{thm:br2019}.
\end{proof}

Thus if $G$ is a group satisfying  the hypotheses of Theorem~\ref{thm:main}, then passing to a finite index subgroup $G_0$, one
often obtains a group
satisfying the hypotheses of Theorem~\ref{thm:br2019}.
In such a case, one can build a $C^1$ action of $G$ on the disjoint union of $n$ copies of $[0,1]$, where
$n=[G:G_0]$, by an analogue of the induced representation of a finite index subgroup. Such an action will permute the components
of this manifold transitively.
This does not contradict Theorem~\ref{thm:main},
since any such action will be outside of a fixed neighborhood of the trivial representation of $G$.

One can produce many fibered $3$--manifold groups, even hyperbolic ones, which are residually torsion--free nilpotent, without using the deep
results of Agol and Wise. Indeed, it suffices to use monodromy maps $\phi$ such that $\phi^*$ is unipotent (i.e.\ has all eigenvalues equal to one).
In this case, the resulting $G$ will always be residually
torsion-free nilpotent~\cite{KoberdaJTA}. In fact, a semidirect product of $\bZ$ with a finitely generated, residually torsion--free 
nilpotent group $H$ will again be residually torsion--free nilpotent if the induced $\bZ$--action on $H^1(H,\bR)$ is unipotent.

It is not true that if a semidirect product $H\rtimes_\psi \bZ$ is residually torsion--free nilpotent then $\psi^*$ is unipotent. 
Indeed, considering a fibered hyperbolic 3-manifold group $\pi_1(T)$ satisfying the
hypotheses of Theorem~\ref{thm:main} and passing to a finite index subgroup which is special (as in Proposition~\ref{prop:virtually special}),
we can obtain a new mapping torus structure
on a finite cover $T_0$ of $T$ with monodromy $\phi_0$, and with a fiber $S_0$ which covers $S$. The action of $\phi_0^*$ on $H^1(S_0,\bR)$
will not be unipotent, since $H^1(S,\bR)$ will naturally sit inside $H^1(S_0,\bR)$ via pullback and will be invariant under $\phi_0^*$.

The regularity assumption in Theorem~\ref{thm:br2019} is subtle. Nonabelian nilpotent groups cannot admit faithful $C^2$ actions on
any compact one--manifold~\cite{PT1976}. Right-angled Artin groups and specialness do not provide any help in producing higher
regularity actions, since in dimension one they almost never admit faithful $C^2$ actions on
compact manifolds~\cite{BKK2016,KKFreeProd2017}. The compactness of the manifold acted upon here is also essential; see~\cite{BKK2014}.

\subsection{General group actions on compact manifolds}

A robust trend in the theory of group actions on manifolds is that ``large" groups should not act on ``small" manifolds. Among the striking
results in this area are the facts that irreducible lattices in higher rank semisimple Lie groups do not admit infinite image $C^1$ actions (and
often even $C^0$ actions)
on compact $1$--manifolds~\cite{BM1999,Ghys1999,Witte1994}. For higher dimensional manifolds, the work of Brown--Fisher--Hurtado
shows that for $n\geq 3$, groups commensurable with $\mathrm{SL}_n(\bZ)$ do not admit faithful $C^1$ actions on $m$--dimensional
compact manifolds for $m<n-2$, and for $m<n-1$ if the actions preserve a volume form~\cite{BFH2016,BFH2017}. They obtain similar
results for cocompact lattices in simple Lie groups.

Lattices in rank one Lie groups often do admit faithful smooth actions on compact one manifolds. By~\cite{BHW2011}, many arithmetic
lattices in $\mathrm{SO}(n,1)$ are virtually special and hence
virtually residually torsion-free nilpotent, which by Theorem~\ref{thm:br2019} furnishes many faithful
$C^1$ actions of such lattices.

McCarthy's result~\cite{McCarthy} furnishes a class of solvable 
groups which admit  no faithful, small $C^1$ actions on compact manifolds whatsoever.
Topologically, her groups arise as fundamental groups of torus bundles over the circle, with no restrictions on the dimension.
Our main result identifies a larger class of such groups, including ones within the much more dimensionally restricted
and algebraically different class of compact $3$--manifolds groups.
For fibered $3$--manifolds groups acting without at least some smallness assumptions, we can only make much weaker statements:

\begin{prop}\label{p:c3}
If $T$ is a closed, hyperbolic, fibered  $3$--manifold, then the universal circle
action of the fundamental group $\pi_1(T)$ on $S^1$ is not topologically conjugate to a $C^3$ action.
\end{prop}

Proposition~\ref{p:c3} follows immediately from the work of Miyoshi~\cite{Miyoshi}. We will deduce Proposition~\ref{p:c3} from a stronger fact (Proposition~\ref{p:c3-2}) in Section~\ref{sec:general} for the convenience of the reader.

There is no hope of establishing a result as sweeping as the Brown--Fisher--Hurtado resolution of many cases of the Zimmer Conjecture
for $3$--manifold groups acting on the circle, even with maximal regularity  assumptions:

\begin{prop}[{e.g.~\cite{Calegari2006GD}}]\label{p:analytic}
There exist finite volume hyperbolic $3$--manifold subgroups of $\mathrm{PSL}_2(\bR)$.
\end{prop}

Any such groups act by projective (and hence analytic) diffeomorphisms on $S^1$. We remark that Proposition~\ref{p:analytic} seems
well-known to experts. We refer the reader to Subsection~\ref{ss:analytic} for a further discussion of  analytic hyperbolic $3$--manifold
group actions  on the circle.

\subsection{Uniqueness of the presentation of $G$}

We remark briefly that if $G=\pi_1(T)$
satisfies the hypotheses of Corollary~\ref{cor:main-3M} then there is an essentially unique homomorphism $G\to \bZ$ whose
kernel is isomorphic to a finitely generated group, and in particular the fibered $3$--manifold structure on $T$ is unique
(see~\cite{Thurston86,Stallings1962}). 
Thus, the induced map $\psi^*$ is canonically defined, and one may therefore speak of \emph{the}
monodromy action. For fibered $3$--manifold groups with first Betti number $b_1>1$ this is no longer the case.

\section{Preliminaries}

In this section, we gather the tools we will need to establish the principal result of this paper.

\subsection{The space of $C^1$ actions of $G$}\label{subsec:space}

Recall that in our notation, $M$ denotes a fixed compact Riemannian manifold. 
We denote by $\Diff^1(M)$ the group of $C^1$--diffeomorphisms of $M$. 
For $f\in\Diff^1(M)$, we will write \[D_xf \colon T_xM\to T_{f(x)}M\] for the Jacobian of $f$.

It will be convenient for us to assume that $M$ is $C^1$ embedded in a Euclidean space $\bR^N$ for some $N\gg 0$.
For our purposes, we only require an embedding, though in principle one could require an isometric embedding
by the Nash Embedding Theorem~\cite{Nash}, for example. We reiterate that an isometric embedding is not necessary for the sequel. 

For brevity, we let $\|X\|$ denote the $\ell^\infty$ norm when $X$ is a function, a vector, a matrix or a tensor.
We replace distances in $M$ by distances in $\bR^N$, and
we equip the Jacobian of a diffeomorphism $f$ of $M$ with the $\ell^\infty$ norm arising from $\bR^N$, which we denote by $\|D_xf\|$. 
Note that if $V\cong\bR^N$  is a vector space equipped with the $\ell^\infty$ norm and $T\in\mathrm{End}(V)$, then we have the estimate
\[\|Tv\|\leq N\|T\|\|v\|\] for all $v\in V$. This estimate is in fact more general. Indeed, sometimes, we will
consider linear maps which are defined on subspaces of $\bR^N$ and which have values in $\bR^N$ (for example, the Jacobian of a
diffeomorphism of $M\subset\bR^N$  as above). In this case, we are still able to write down a matrix (which will no longer be square),
that represents this linear map. We will define the supremum norm of such a matrix by  taking the maximum  of the absolute
values of the entries, in which case the same norm estimate holds as for an endomorphism of $V$.
We will make essential use of this estimate in the sequel.

We define the \emph{$C^1$--metric} on $\Diff^1(M)$ by
 \[d(f,g)=\abss{f-g} + \sup_{x\in M} \abss{D_x f - D_x g},\]
 where all these distances and norms are now interpreted in the ambient Euclidean space.

If $G$ is generated by a finite set $S$, then we may define a metric $d_S$ on $\Hom(G,\Diff^1(M))$ via
\[d_S(\rho,\rho')=\max_{s\in S}d(\rho(s),\rho'(s)).\]
This metric $d_S$ determines the \emph{$C^1$--topology} of $\Hom(G,\Diff^1(M))$, and this topology is independent of the choice of the generating set $S$.

For an arbitrary group $G$, we will write $\rho_0\in\Hom(G,\Diff^1(M))$ for the trivial representation of $G$. 
We see that in order to prove Theorem~\ref{thm:main}, it suffices to find some $\epsilon>0$ such that every representation $\rho\in\Hom(G,\Diff^1(M))$ satisfying $d_S(\rho,\rho_0)<\epsilon$  maps $H$ to the identity $1$.

\subsection{Hyperbolic monodromies}

Here, we recall some basic facts from linear algebra of hyperbolic automorphisms of a real vector space. Let $V$ be a
$d$--dimensional vector space over $\bR$, and let $\|\cdot\|_d$ be a fixed norm on $V$. When $A\in\mathrm{GL}(V)$, we say that
$A$ is \emph{hyperbolic} if every eigenvalue of $A$ has modulus different from one.

\begin{lem}\label{lem:hyperbolic}
Let $A\in\mathrm{GL}(V)$ be a hyperbolic automorphism. Then there is an $A$--invariant splitting $V= E^-\oplus E^+$ and a positive integer $p_0$ such that the following conclusions hold
for all $p\ge p_0$:

\begin{enumerate}
\item
if $v\in E^-$ then \[\|A^pv\|_d\leq\frac{1}{2}\|v\|_d;\]
\item\label{i2:hyp}
if $v\in E^+$ then \[\|A^pv\|_d\geq 2\|v\|_d.\]
\end{enumerate}
\end{lem}

We omit the proof of the lemma, which is well--known; see~\cite[Chapter 1]{KH1995} for instance.
As is standard from dynamics, $E^-$ and $E^+$ are the
\emph{stable} and \emph{unstable} subspaces of $V$ associated to $A$.
In the sequel, we will use the  notation $\pi_+ $ and $\pi_-$ to denote projections $V\to E^+$ and $V\to E^-$ with kernels $E^-$ and $E^+$
respectively. Observe that invariance of the splitting implies that $A$ commutes with each projection $\pi_{+}$ and $\pi_-$.

\subsection{Approximate linearization}

A fundamental tool for proving Theorem~\ref{thm:main} is the following result of Bonatti~\cite{BonattiProches,BMNR2017},
which arose as an interpretation of
Thurston Stability~\cite{Thurston1974Top}, and which we refer to as \emph{approximate linearization}.

\begin{lem}\label{lem:bonatti}
Let $M$ be a compact manifold, let $\eta>0$, and let $k\in\bN$.
Then there exists a neighborhood of the identity $\VV\sse \Diff^1(M)$ such that for all points $x\in M$, for all diffeomorphisms
$f_1,\ldots,f_k\in \VV$
and for all $\epsilon_1,\ldots,\epsilon_k\in \{-1,1\}$,
we have the following:
\begin{equation*}\label{eq:tstab}
\left \|
f_k^{\epsilon_k}\circ \cdots\circ f_1^{\epsilon_1}(x) - x
- \sum_{i=1}^k \epsilon_i (f_i(x)-x)\right \|\le \eta \max_{i=1,\ldots,k} \left\|f_i(x)-x\right \|.
\end{equation*}
\end{lem}

Throughout the rest of this paper, we will often suppress the notation $\rho\in \Hom(G,\Diff^1(M))$ and just write 
$gx=g(x)=\rho(g)(x)$ for $g\in G$ and $x\in M$.
We define a \emph{displacement vector} for $g$ at $x$ as
\[\Delta^\rho_x(g) := \rho(g)(x) - x,\]
regarded as an $N$--dimensional row vector. Here, we remind the reader that $M$ is embedded  as a submanifold of $\bR^N$, so that  the
displacement vector becomes a vector in $\bR^N$. Admittedly, the displacement vector depends on the choice of
embedding, though this does not matter since we will ultimately be interested in whether or not it vanishes.

More generally, if $B=\{b_1,\ldots,b_n\}\sse G$ is a finite subset then we define an $n\times N$ matrix
\[\Delta^\rho_x(B) := \left( \Delta_x^\rho(b_i)\right)_{1\le i\le n}.\]
We often write $\Delta_x$ for $\Delta^\rho_x$ when the meaning is clear.
 Then the above lemma asserts that
\begin{equation*}\label{eq:tstab2}
\abss*{
\Delta_x(g_k^{\epsilon_k}\circ \cdots\circ g_1^{\epsilon_1})- \sum_{i=1}^k \epsilon_i \Delta_x(g_i)}\le\eta
\abss*{\Delta_x(\{g_1,\ldots,g_k\})},
\end{equation*}
in the case when $g_i\in G$ and $\rho(g_i)\in \VV$. Here, we remind  the reader that we always  use the supremum norm.

\subsection{First homology and cohomology groups}
We briefly recall for the reader unfamiliar with group homology
that the first homology group of a group $H$ is given by the abelianization \[H_1(H,\bZ)=H/[H,H].\]
When $R\in\{\bZ,\bR\}$, the first cohomology group $H^1(H,R)$ coincides with the abelian group of homomorphisms from $H$ to $R$.
In particular, $H^1(H,\bZ)$ is a free abelian group of the same rank as $H_1(H,\bZ)$.

\section{Proof of Theorem~\ref{thm:main}}
We are now ready to give a proof of Theorem~\ref{thm:main}.
For this, we will fix an automorphism $\psi\in \Aut(H)$ such that $G$ can be written as 
\[G = \form{ H,t \mid tht^{-1}=\psi(h)\text{ for all }h\in H}.\]

\subsection{Reducing to homologically independent generators}
We first establish Lemma~\ref{l:red} below, which will say that we may more or less assume that $H$ is finitely generated and free abelian.

Let $d\ge 1$ be the rank of $H^1(H,\bZ)$. We can find a finite generating set \[S=S_0\sqcup S_1\]
of $H$ such that all of the following hold.
\begin{itemize}
\item The image of $S_0$ in $H_1(H,\bZ)=H/[H,H]$ is a basis for the free part.
\item The image of each element in $S_1$ is torsion or trivial in $H_1(H,\bZ)$.
\end{itemize}

We pick $K\ge2$ so that $\tau^K=0$ for all \[\tau\in \ker\{H_1(H,\bZ)\to H_1(H,\bR)=H_1(H,\bZ)\otimes_{\bZ}\bR\},\]
where the map between
the homology  groups is   the tensoring map.
We enumerate $S_0=\{s_1,s_2,\ldots,s_d\}$, and regard $S_0$ as an ordered set.
Let $A:=(\alpha_{ij})$ be the matrix of the hyperbolic linear automorphism \[\psi^*\co H^1(H,\bZ)\to H^1(H,\bZ)\]
with respect to the basis which is dual to $S_0$, viewed as real homology classes. 
The action $\psi_*$ on  $H_1(H,\bZ)$ is then given by
the transpose $(\alpha_{ji})$. In this case, we can write each $\psi(s_j)$ as 
\begin{equation}\label{eq:conj_psi}
\psi(s_j)=ts_jt^{-1}=\prod_{i=1}^d s_i^{\alpha_{ji}}\tau_j
\end{equation}
for some element $\tau_j\in H$ such that $\tau_j^K\in[H,H]$.
It will be convenient for us to define the subset
\[
S':=  \{u^K\co u\in S_1\cup \{\tau_1,\ldots,\tau_d\}\}\sse[H,H].\]
Observe that each element $h\in [H,H]$ can be expressed as a product of
commutators in $S$. It follows that $h$ can be expressed as a \emph{balanced} word in $S$, which is to say that all generators in $S$ occur with exponent sum zero.
Since $S'\sse [H,H]$, we can find an integer $k_0\ge K$ such that 
 every element in $S'$ is a balanced word of length at most $k_0$ in $S$.  Recall our convention $\|A\|:=\max_{i,j}|\alpha_{ij}|$. We set
\begin{equation}\label{eq:fix_k}
k:=k_0+d\|A\|.
\end{equation}
\begin{lem}\label{l:red}
Let $0<\eta<1$.
Then there exists a neighborhood $\UU\sse\Hom(G,\Diff^1(M))$ of the trivial representation $\rho_0$ such that each of the following relations hold for all representations $\rho\in\UU$ and points $x\in M$.
\begin{enumerate}
\item \label{i1:red}
 $\abss{\Delta^\rho_x(S')}\le\eta \abss{\Delta^\rho_x(S)}$;
\item\label{i2:red} 
$\abss{\Delta^\rho_x(S_1\cup\{\tau_1,\ldots,\tau_d\})}\le\eta \abss{\Delta^\rho_x(S)}$;
\item\label{i3:red}
$\abss{\Delta^\rho_x(S)}= \abss{\Delta^\rho_x(S_0)}$;
\item\label{i4:red}
$\abss*{
\Delta^\rho_x(\psi(S_0))-A \Delta^\rho_x(S_0)}
\le
2\eta \abss*{\Delta^\rho_x(S_0)}$.
\end{enumerate}
\end{lem}

\begin{proof}
Let $k$ be defined as in \eqref{eq:fix_k}.
We have an identity neighborhood  $\VV\sse\Diff^1(M)$ furnished 
by Lemma~\ref{lem:bonatti} for $\eta$ and $k$.
We define $\mathcal{U}$ by
\[\mathcal{U}=\left \{\rho\in\Hom(G,\Diff^1(M))\co \rho(S\cup\{\tau_1,\ldots,\tau_d\})\sse\VV\right \}.\]
We now fix $\rho\in\UU$, and we suppress $\rho$ from the notation by writing
 $g(x):=\rho(g)(x)$. Similarly, we write $\Delta_x(g):=\Delta_x^\rho(g)$. So, $\Delta_x(g)$ will be thought of as a function of the group element
 $g$, and which depends on $x$ as well.

\eqref{i1:red} Let $u\in S'$, so that $u$ can be expressed as a balanced word in  $S$ with length at most $k_0<k$. We see from Lemma~\ref{lem:bonatti} that
\[\abss{\Delta_x(u)}\le\eta\abss{\Delta_x(S)}.\]
This proves part \eqref{i1:red}. 

\eqref{i2:red}
Let $u\in S_1\cup\{\tau_1,\ldots,\tau_d\}$.
Since $u\in \VV$ by assumption, we again use Lemma~\ref{lem:bonatti} to see that
\[
\abss*{\Delta_x(u^K) - K \Delta_x(u)} \le\eta \abss*{\Delta_x(u)}.\]
Using the triangle inequality and part \eqref{i1:red}, we see that
\begin{align*}
K\abss*{\Delta_x(u)}
\le \abss*{\Delta_x(u^K)} +\eta\abss*{\Delta_x(u)}
\le \eta \abss{\Delta_x(S)}+\eta\abss*{\Delta_x(u)}.
\end{align*}
Since $K\ge2$, we obtain the desired conclusion as
\[
\abss*{\Delta_x(u)}\le \frac{\eta}{K-\eta}\abss{\Delta_x(S)}
 \le \eta \abss{\Delta_x(S)}.\]

Part \eqref{i3:red} is obvious from the previous parts. For part \eqref{i4:red}, 
let us pick an arbitrary $s_j\in S_0$.
From the expression \eqref{eq:conj_psi} for $\psi(s_j)=ts_jt^{-1}$ and from
Lemma~\ref{lem:bonatti},
we can deduce that
\[
\abss*{
\Delta_x(\psi(s_j))-\sum_{i=1}^d\alpha_{ji} \Delta_x(s_i)
-\Delta_x(\tau_j)
}
\leq \eta\abss{\Delta_x(S\cup\{\tau_j\})}=
 \eta\abss{\Delta_x(S_0)}.\]
The triangle inequality and the second and third parts of the lemma imply the 
conclusion of part \eqref{i4:red}.
\end{proof}

\subsection{McCarthy's Lemma}
Retaining previous notation, we have a group $G$ presented as $H\rtimes_\psi \form{t}$.
Another ingredient for the proof of the main theorem is the following lemma, which was proved by McCarthy~\cite[Lemmas 4.1 and 4.2]{McCarthy} in the case when $H$ is abelian:

\begin{lem}[{cf.\ Lemmas 4.1 and 4.2 of~\cite{McCarthy}}]\label{l:mccarthy}
For all $\eta\in(0,1/3)$, there exists
 a neighborhood $\mathcal{U}\sse\Hom(G,\Diff^1(M))$ of the trivial representation $\rho_0$ such that whenever $\rho\in\mathcal{U}$ and $x\in M$ we have 
\[
\abss*{
\Delta^\rho_{\rho(t^{-1})(x)}(S_0)-A \Delta^\rho_x(S_0)}
\le \eta \abss*{\Delta^\rho_x(S_0)}.\]
\end{lem}
Roughly speaking, under the above hypothesis if we denote the displacement matrix of $S_0$ at $x$ as $v$, then $Av$ will be near from the displacement matrix of $S_0$ at $t^{-1}x$. Thus, one can apply hyperbolic dynamics to estimate the change of displacement matrices as points are moved under iterations of $t^{-1}$: \[x\mapsto t^{-1}x\mapsto t^{-2}x\mapsto\cdots\mapsto t^{-n}x\mapsto\cdots\]

Since McCarthy's arguments concerned the case where $H$ abelian and hence do not apply in this situation, let us reproduce proofs here
which work for general groups.

\bp[Proof of Lemma~\ref{l:mccarthy}]
Fix $\eta'\in(0,\eta)$, which will be made  explicit later.
We pick a sufficiently small neighborhood $\mathcal{U}\sse\Hom(G,\Diff^1(M))$ of $\rho_0$, which is at least as small as the neighborhood $\UU$ in Lemma~\ref{l:red} for this choice of $\eta'$. 
We let $\rho\in\UU$, and again suppress the notation $\rho$ in expressions. 
We also fix $x\in M$, and set $y:=t^{-1}x$. 

Suppose we have $s\in S_0$. From the definition of the derivative, we have that 
\[\Delta_x(\psi(s))
=
\Delta_{ty}(tst^{-1})=ts(y)-t(y)=D_yt(\Delta_y(s))+o(\|\Delta_y(s)\|).\]
Replacing $\mathcal{U}$ by a smaller neighborhood if necessary, we may assume that (with a slight abuse of notation)
\[o(\|\Delta_y(s)\|)<\eta'\|\Delta_y(s)\|\]
in norm,
 and that $N\abss{D_yt-1}\le \eta'$, where $1$ denotes the identity map, and $N$ is the dimension of the Euclidean space where $M$ is embedded.
It then follows that 
\begin{equation}\label{eq:taylor}
\abss{\Delta_x(\psi(s))-\Delta_y(s)}
\le N \abss{D_y t - 1}\cdot \abss{\Delta_y(s)} +o(\Delta_y(s))
\le 2\eta'\abss{\Delta_y(s)}.\end{equation}
Here, we are using the $\ell^{\infty}$ norm estimate \[\|Tv\|\leq N\|T\|\|v\|\] for arbitrary vectors $v$ and linear maps $T\colon  \bR^N\to \bR^N$.

 Applying the triangle inequality, Lemma~\ref{l:red}~\eqref{i4:red} and \eqref{eq:taylor}, we deduce that
\begin{equation}\label{eq:x1}
\abss*{\Delta_y(S_0)-A\Delta_x(S_0)}
\le 2\eta' \abss{\Delta_y(S_0)}+ \abss*{\Delta_x(\psi(S_0))-A\Delta_x(S_0)}
\le2\eta' \left (\abss{\Delta_y(S_0)}+ \abss{\Delta_x(S_0)}\right ).\end{equation}
From the inequality \eqref{eq:x1}, we note that
\begin{equation}\label{eq:x2}
(1-2\eta')\abss*{\Delta_y(S_0)}
\le
\abss*{A\Delta_x(S_0)}+2\eta'  \abss{\Delta_x(S_0)}
\le
(d\|A\|+2\eta')\abss{\Delta_x(S_0)}.\end{equation}
We will now choose $\eta'\in(0,\eta)$ sufficiently small so that 
\[
\left(\frac{d\|A\|+2\eta'}{1-2\eta'}+1\right)\cdot 2\eta'
\le\left(\frac{d\|A\|+2/3}{1/3}+1\right)\cdot 2\eta'\le\eta.\]
Combining inequalities \eqref{eq:x1} and \eqref{eq:x2} we obtain the desired conclusion as 
\[
\abss*{\Delta_y(S_0)-A\Delta_x(S_0)}
\le
2\eta'\left( \frac{d\|A\|+2\eta'}{1-2\eta'}+1\right)\abss{\Delta_x(S_0)}
\le\eta\abss{\Delta_x(S_0)}.\qedhere\]
\ep

\subsection{Finishing the proof}
We can now complete the proof of the main result.

\begin{proof}[Proof of Theorem~\ref{thm:main}]
Let $\rho$ be sufficiently near from $\rho_0$.
By Lemma~\ref{l:red}~\eqref{i3:red},  it suffices for us to prove that the $d\times N$ matrix
$\Delta_x(S_0)$ is equal to $0$ for all points $x\in M$.

The hyperbolic automorphism $\psi^*$ on $H^1(H,\bZ)$ induces an invariant splitting 
\[
\bR^d=\oplus_{i=1}^d \bR s_i = E^+\oplus E^-,\]
as in Lemma~\ref{lem:hyperbolic}.
We may assume $p_0=1$ in that lemma after replacing $\psi^*$ by a sufficiently large power; this is the same as passing to the kernel of the natural map $G\to\bZ/p\bZ$ given by reducing $G/H$ modulo $p$.

Let us pick a point $x\in M$ such that the quantity
\[
\max \left( \abss{\pi_+\Delta_z(S_0)},\abss{\pi_-\Delta_z(S_0)}\right)\]
is itself maximized at $z=x$. 
Here, $\pi_{\pm}$ is regarded as a map from $\oplus_{i=1}^N \bR^d$ to $\oplus_{i=1}^N E^{\pm}$.

For a proof by contradiction, we will suppose that this maximum is nonzero. We may further assume the maximum occurs for the unstable direction. Since the stable and unstable subspaces of a hyperbolic matrix are symmetric under inversion, the case where the maximum is in the stable direction is analogous. 

Let us choose $\eta\in(0,1/3)$ and a neighborhood $\UU\sse\Hom(G,\Diff^1(M))$ so that the conclusion of Lemma~\ref{l:mccarthy} holds for $\rho\in\UU$. With this choice,
using also the contraction property of $\pi_+$, we estimate
\begin{align*}
\abss{\pi_+\Delta_{t^{-1}x}(S_0)-
\pi_+A\Delta_{x}(S_0)}
& \le 
\abss{\Delta_{t^{-1}x}(S_0)-
A\Delta_{x}(S_0)}
\le \eta \abss{\Delta_x(S_0)}\\
&\le \eta \left(\abss{\pi_+\Delta_x(S_0)}+\abss{\pi_-\Delta_x(S_0)}\right)
\le 2\eta \abss{\pi_+\Delta_x(S_0)}.\end{align*}
On the other hand, applying the triangle inequality and Lemma \ref{lem:hyperbolic} \eqref{i2:hyp} we have
\[
\abss{\pi_+\Delta_{t^{-1}x}(S_0)-
A\pi_+\Delta_{x}(S_0)}
\ge 
\abss{A\pi_+\Delta_{x}(S_0)}
-\abss{\pi_+\Delta_{t^{-1}x}(S_0)}
\ge 2\abss{\pi_+\Delta_{x}(S_0)}-\abss{\pi_+\Delta_{t^{-1}x}(S_0)}.\]
Combining the above chains of inequalities, and using that $A\pi_+=\pi_+A$, we obtain 
\[
\abss{\pi_+\Delta_{t^{-1}x}(S_0)}
\ge 2(1-\eta)\abss{\pi_+\Delta_{x}(S_0)}
> \abss{\pi_+\Delta_{x}(S_0)}.\]
This contradicts the maximality of our choices. \ep

\section{General group actions and questions}\label{sec:general}

As remarked in the introduction, there is no hope of ruling out
highly regular faithful actions of $3$--manifold groups on low dimensional manifolds. Thus, Theorem~\ref{thm:main} can be viewed
as a local rigidity phenomenon of $\Hom(G,\Diff^1(M))$ near the trivial representation $\rho_0$ rather than as a global statement about this space of actions. In
this section we discuss actions of $3$--manifold groups on the circle which are not small, and thus are much less constrained.

\subsection{Universal circle actions}

First, we show that for certain types of faithful actions of $3$--manifold groups, some regularity constraints persist. Let $T$ be a fibered
$3$--manifold with closed, orientable fiber $S$ and monodromy $\psi\in\Mod(S,p)$. We assume that $\chi(S)<0$.
Here, we have equipped $S$ with a basepoint $p$, and we assume that elements of $\Mod(S,p)$
preserve $p$, as do isotopies between them.

We have that the fundamental group $\pi_1(S)$ naturally sits in $\Mod(S,p)$ as the kernel of the homomorphism $\Mod(S,p)\to\Mod(S)$ which forgets the basepoint
$p$~\cite{FM2012}. The short exact sequence \[1\to\pi_1(S)\to\Mod(S,p)\to\Mod(S)\to 1\] is known as the Birman Exact Sequence. The
mapping class group $\Mod(S,p)$ has a natural
faithful action on $S^1$ by homeomorphisms, known as Nielsen's action (see~\cite{CassonBleiler}).
This action of $\Mod(S,p)$ is not conjugate to a $C^1$ action, and even after passing to finite index subgroups it is known not to be
conjugate to a $C^2$
action~\cite{FF2001,BKK2016,KKFreeProd2017,Parwani2008,MannWolff18}. 
Moreover, this action is not  absolutely continuous, as can be easily seen from Proposition~\ref{p:normalizer} below.
However, one can topologically conjugate Nielsen's action to a
bi-Lipschitz one; this is a general fact for countable groups acting on the circle~\cite{DKN2007}.
We remark that Nielsen's action, as it is constructed by extensions of quasi-isometries of $\mathbb{H}^2$ to $S^1$, enjoys a regularity
property known as quasi--symmetry. See~\cite{CassonBleiler,HT85,FM02}.

If $\psi\in\Mod(S,p)$ then the conjugation action of $\psi$ on the group \[\pi_1(S)=\ker\left\{\Mod(S,p)\to\Mod(S)\right\}\]
makes the group
$\langle \psi,\pi_1(S)\rangle$ isomorphic to $\pi_1(T)$. We thus obtain an action, which is called the \emph{universal circle action} of $\pi_1(T)$
(see~\cite{Calegari2007}).
While it follows that $\pi_1(T)$ admits a natural faithful action on $S^1$ by absolutely continuous homeomorphisms,
the higher
 regularity properties of this action are somewhat mysterious.
 
 We now give a proof of Proposition~\ref{p:c3}, which asserts that this action is not topologically conjugate
to a $C^3$ action. As stated in the introduction, this result is known from the work of Miyoshi.
The proof of Proposition~\ref{p:c3} given in~\cite{Miyoshi} follows similar lines to the argument given here,
and is easily implied by the following two results:

\begin{prop}\label{p:normalizer}
	Let $S$ be a closed surface and $\rho:\pi_1(S)\to \PSL_2(\bR)$ be a faithful discrete representation.
	Then the normalizer of $\rho(\pi_1(S))$ in $\Homeo^{\mathrm{ac}}(S^1)$ is a discrete subgroup of $\PSL_2(\bR)$
	which contains $\rho(\pi_1(S))$ as a finite-index subgroup.
\end{prop}

\bp
Let $g$ be an absolutely continuous homeomorphism of the circle which normalizes $\rho(\pi_1(S))$.
Then by an argument originally due to Sullivan (see \cite[Proposition III.4.1]{Ghys1985}), we see that $g$ is actually contained in
$\PSL_2(\bR)$. Ghys gives a relatively simple argument under the assumption of $C^1$ conjugacy,
which in turn suffices for Proposition \ref{p:c3}.
In this case, all derivatives of hyperbolic elements at their fixed points must be preserved by the conjugacy.
In other words, the marked length spectrum associated with the Fuchsian group $\rho(\pi_1(S))$ is invariant,  and
thus the isometry class of the corresponding hyperbolic surface is preserved.

Now, it follows from standard facts about Zariski dense subgroups of simple Lie groups
 that the normalizer of a Fuchsian group in $\PSL_2(\bR)$ is necessarily Fuchsian,
 and contains the original Fuchsian group with finite index. Indeed, suppose $\Gamma<\PSL_2(\bR)$ is discrete and let
 $\{g_i\}_{i\in \bN}\subset\PSL_2(\bR)$ normalize $\Gamma$. Suppose furthermore that $g_i\to 1$ as $i\to\infty$. Then it is not difficult
 to show that $g_i$ must centralize $\Gamma$ for $i$ sufficiently large. If $\Gamma$ is Zariski dense then  $g_i$ is the
 identity for $i$ sufficiently large, so  that the normalizer of $\Gamma$ is again discrete. If $\Gamma$ is cocompact then its normalizer must
  contain $\Gamma$ with finite index. See~\cite[Theorem 2.3.8]{KatokBook} for more details. The conclusion of the proposition now  follows.
\ep

Proposition~\ref{p:normalizer} implies the following: let $\psi\in\Mod(S)$ be pseudo-Anosov,
and let $\yt{\psi}$  be in the preimage of $\psi$ under the
canonical map $\Mod(S,p)\to\Mod(S)$ given by deleting the marked point. Then $\yt\psi$ 
fails to act by an absolutely continuous homeomorphim on $S^1$ under Nielsen's action of $\Mod(S,p)$ on $S^1$.

The following result is known as Ghys' differentiable rigidity of Fuchsian actions~\cite{GhysIHES93}.

\begin{thm}\label{thm:ghys93}
	Let $S$ be a closed surface and let $\rho:\pi_1(S)\to \Diff^r(S^1)$
	for $r\ge 3$ be a representation which is topologically conjugate to a Fuchsian subgroup of $\PSL_2(\bR)$.
	Then $\rho$ is conjugate to a Fuchsian subgroup of $\PSL_2(\bR)$ by a $C^r$ diffeomorphism.
\end{thm}

Proposition~\ref{p:c3} is an immediate consequence of the following,
which in turn is an obvious corollary of Proposition~\ref{p:normalizer} and Theorem~\ref{thm:ghys93}.

\begin{prop}\label{p:c3-2}
Let $T$ be a hyperbolic fibered 3-manifold with a closed fiber $S$.
If an action
\[\rho\co \pi_1(T)=\pi_1(S)\rtimes\form{t}\to\Homeo_+(S^1)\] 
has the property that $\rho(\pi_1(S))$ is topologically conjugate to a Fuchsian subgroup of
$\PSL_2(\bR)$, then either $\rho(\pi_1(S))\not\le\Diff^3_+(S^1)$ or $\rho(t)$ is not absolutely continuous. 
\end{prop}

We remark that universal circle actions enjoy a strong $C^0$ rigidity property, namely that actions in the same connected component
of the representation variety of $\pi_1(T)\to\Homeo_+(S^1)$ are semi-conjugate
 to the standard action~\cite{BM2019}.

\subsection{Analytic actions}\label{ss:analytic}

Finally, we discuss faithful analytic actions of fibered $3$--manifold groups on $S^1$. By Agol's resolution of the virtual fibering
conjecture~\cite{Agol2013}, we have that every hyperbolic $3$--manifold virtually fibers over the circle. Thus, if a subgroup
$\Gamma<\mathrm{PSL}_2(\bC)$ is discrete (i.e.\ a Kleinian group) with finite covolume, then $\Gamma$ has a finite index subgroup which
is $\pi_1(T)$ for some fibered $3$--manifold $T$. Now, if the matrix entries of $\Gamma$ are contained in a number field $K\supset\bQ$
such that $K$ has a real
place (i.e.\ a Galois embedding $\sigma\colon K\to\bC$ such that $\sigma(K)\sse\bR$), then $\Gamma$ can be identified
with a subgroup of $\PSL_2(\bR)$.

Therefore, in order to establish Proposition~\ref{p:analytic}, it suffices to produce such a Kleinian group. If $\Gamma$ has matrix entries
in a field $K$ of odd degree over $\bQ$ then $K$ has at least one real place, since the number of complex places is even.
Many such arithmetic Kleinian groups of finite covolume (and even cocompact ones such as the fundamental group of the
Weeks  manifold) exist; see \cite[Section 13.7]{MR2003book}, for example. Note that since  any discrete  subgroup of $\PSL_2(\bR)$ is
virtually free or a closed surface group, a finite volume hyperbolic $3$--manifold group cannot occur as a discrete subgroup of $\PSL_2(\bR)$.

\subsection{Questions}

There are several natural questions which arise from the discussion in this paper.

\begin{que}[J. Souto]
Let $T$ be a fibered $3$--manifold and let $G=\pi_1(T)$. Is there a finite index subgroup $G_0<G$ such that $G_0<\Diff^2(I)$? What about
$G_0<\Diff^{\infty}(I)$?
\end{que}

In~\cite{MS2018}, Marquis and Souto constructed a faithful $C^{\infty}$ action of closed orientable surface groups, for genus $g\geq 2$, on the unit
interval.

\begin{que}\label{q:uni-circle}
Is the universal circle action of a fibered $3$--manifold group topologically conjugate to a $C^1$ action?
\end{que}

In other words, Question~\ref{q:uni-circle} asks if we can replace the $C^3$ conclusion in Proposition~\ref{p:c3} with a $C^1$ conclusion. Observe that Ghys' differentiable rigidity of Fuchsian actions does not hold in lower regularity (at least less than $C^2$): for arbitrary $\alpha<1$, there are $C^{1+\alpha}$ actions of $\pi_1(S)$  that are
$C^0$ conjugate to a Fuchsian action, but that are not conjugate to a Fuchsian action by an absolutely continuous homeomorphism; see \cite{HurderKatok}. Other instances of this phenomenon arise from the theory of Hitchin representations~\cite{MR3749423}. A first attempt to answer Question~\ref{q:uni-circle} would be to investigate if the analogue of Proposition \ref{p:normalizer} holds for these actions: do they admit a $C^1$ normalizer which is not a finite extension of the image of $\pi_1(S)$?

\section*{Acknowledgements}
The authors thank Nicolas Tholozan and Kathryn Mann for useful comments.
CB and MT are partially supported by the project ANR Gromeov (ANR-19-CE40-0007).
ShK is supported by Samsung Science and Technology Foundation (SSTF-BA1301-51) and by a KIAS Individual Grant (MG073601) at Korea Institute for Advanced Study. 
TK is partially supported by an Alfred P. Sloan Foundation Research Fellowship, and by NSF Grant DMS-1711488.
MT is partially supported by PEPS -- Jeunes Chercheur-e-s --  2019 (CNRS),  the project ANER Agroupes (AAP 2019 R\'egion Bourgogne--Franche--Comt\'e), and the project Jeunes
G\'eom\'etres of F. Labourie (financed by the Louis D. Foundation).
ShK and MT thank the organizers of the meeting \emph{International Conference On Dynamical Systems} at SUSTech in 2018, where this project was initiated.
TK and MT are grateful to the Korea Institute for Advanced Study for its hospitality while part of this research was completed.


\bibliographystyle{amsplain}
\bibliography{ref}

\end{document}